\documentclass[sigconf,authorversion=true,screen=true]{acmart}

\usepackage{booktabs} 

\usepackage{graphicx}          
\graphicspath{ {./Figures/}} 

\usepackage[english]{babel}
\usepackage[utf8]{inputenc}
\usepackage[inline]{enumitem}
\usepackage{tikz}
\usepackage{amsmath}
\usepackage{amsfonts}
\usepackage{multicol}
\usepackage[noend]{algpseudocode}
\usepackage{algorithm}
\usepackage{algorithmicx}
\usepackage{amssymb}
\usepackage{bm}
\usepackage{mathtools}
\usepackage{calc}
\usepackage{subcaption}
\usepackage{pgfplots}
\usepackage{pgfplotstable}
\usepackage{wrapfig}
\usepackage{tikz}
\usepackage{tikzscale}
\usetikzlibrary{calc,decorations.pathreplacing,decorations.text,positioning,patterns,topaths,external,3d,spy,pgfplots.statistics}
\usepgfplotslibrary{groupplots,fillbetween}
\pgfplotsset{compat=1.13}

\definecolor{blue}{rgb}{0.25,0.25,0.75}
\colorlet{someblue}{blue!70}
\colorlet{lightblue}{blue!55}
\colorlet{verylightblue}{blue!30}
\definecolor{greyblue}{rgb}{0.36,0.36,0.5}
\colorlet{lightgreyblue}{greyblue!55}
\colorlet{red}{red!90!black}
\colorlet{lightred}{red!65}
\colorlet{verylightred}{red!40}
\colorlet{green}{green!80!black}
\colorlet{lightgreen}{green!65}
\colorlet{verylightgreen}{green!20}
\colorlet{lightgray}{gray!65}

\usepackage{accents}

\DeclareMathOperator*{\esssup}{ess\, sup}
\newcommand{\RP}{\mathbf{RP}}
\newcommand{\RPapprox}{\widetilde{\RP}}
\newcommand{\SP}[1]{\mathbf{SP}[#1]}
\newcommand{\SPapprox}[1]{\widetilde{\mathbf{SP}}[#1]}
\newcommand{\SPopt}[1]{\mathbf{H}^\star[#1]}
\newcommand{\SPapproxoptH}[1]{\widetilde{\mathbf{H}}^\star[#1]}
\newcommand{\SPapproxoptd}[1]{\greekbf{\delta}^\star[#1]}
\newcommand{\Wapprox}[1]{\setbf{W}[#1]}
\newcommand{\RPlin}[1]{\overbar{\RP}[#1]}
\newcommand{\RPlinoptH}[1]{\overbar{\mathbf{H}}^\star[#1]}
\newcommand{\RPlinoptd}[1]{\overbar{\greekbf{\delta}}^\star[#1]}
\newcommand{\slimcdot}{\!\cdot\!}
  
\usepackage{ifamath}
\usepackage[]{units}
\usepackage{xfrac}

\usepackage{ifthen}
\ifthenelse{\isundefined{\transp}}{\newcommand{\transp}{\top}}{}

\setcopyright{none}

\acmDOI{}

\acmConference[HSCC '18]{HSCC '18: 21st International Conference on Hybrid Systems: Computation and Control (part of CPS Week)}{April 11--13, 2018}{Porto, Portugal}
\acmBooktitle{HSCC '18: 21st International Conference on Hybrid Systems: Computation and Control (part of CPS Week), April 11--13, 2018, Porto, Portugal}
\acmYear{2018}
\copyrightyear{2018}
\acmDOI{10.1145/3178126.3178136}
\acmISBN{978-1-4503-5642-8/18/04}

\begin{document}
\title{From Uncertainty Data to Robust Policies for Temporal Logic Planning}


\author{Pier Giuseppe Sessa}
\affiliation{%
 \institution{Automatic Control Laboratory, ETH Zurich}
 \streetaddress{Physikstrasse 3}
 \city{Zurich} 
 \state{Switzerland} 
 \postcode{CH-8092}
}
\email{sessap@student.ethz.ch}
\orcid{0000-0001-8986-8815}
\authornote{These authors contributed equally to this work.}

\author{Damian Frick}
\affiliation{%
 \institution{Automatic Control Laboratory, ETH Zurich}
 \streetaddress{Physikstrasse 3}
 \city{Zurich} 
 \state{Switzerland} 
 \postcode{CH-8092}
}
\email{dafrick@control.ee.ethz.ch}
\orcid{0000-0002-9355-4189}
\authornotemark[1]

\author{Tony A. Wood}
\affiliation{%
 \institution{Automatic Control Laboratory, ETH Zurich}
 \streetaddress{Physikstrasse 3}
 \city{Zurich} 
 \state{Switzerland} 
 \postcode{CH-8092}
}
\email{woodt@control.ee.ethz.ch}
\orcid{0000-0002-3928-3306}

\author{Maryam Kamgarpour}
\affiliation{%
 \institution{Automatic Control Laboratory, ETH Zurich}
 \streetaddress{Physikstrasse 3}
 \city{Zurich} 
 \state{Switzerland} 
 \postcode{CH-8092}
}
\email{mkamgar@control.ee.ethz.ch}
\orcid{0000-0003-0230-3518}
\authornote{The work of M. Kamgarpour is gratefully supported by Swiss National Science Foundation, under the grant SNSF 200021\_172782}

\begin{abstract}
	We consider the problem of synthesizing robust disturbance feedback policies for systems performing complex tasks. We formulate the tasks as linear temporal logic specifications and encode them into an optimization framework via mixed-integer constraints. Both the system dynamics and the specifications are known but affected by uncertainty. The distribution of the uncertainty is unknown, however realizations can be obtained. We introduce a data-driven approach where the constraints are fulfilled for a set of realizations and provide probabilistic generalization guarantees as a function of the number of considered realizations. We use separate chance constraints for the satisfaction of the specification and operational constraints. This allows us to quantify their violation probabilities independently. We compute disturbance feedback policies as solutions of mixed-integer linear or quadratic optimization problems. By using feedback we can exploit information of past realizations and provide feasibility for a wider range of situations compared to static input sequences. We demonstrate the proposed method on two robust motion-planning case studies for autonomous driving.
\end{abstract}

%
%



\maketitle

\sloppy
\section{Introduction}
Increased automation in transportation, energy systems and industrial manufacturing necessitates autonomously performing increasingly complex tasks. 
These tasks have to be completed safely and reliably despite the presence of uncertainties.
Linear temporal logic (LTL) \cite{pnueli1977} is a formal language that enables specifying such complex tasks as a combination of simpler tasks. To achieve this, LTL combines propositional logic with temporal operators.
In this paper we develop control policies for dynamical systems and specifications affected by uncertainty. These policies are computed from samples of the uncertainty and we provide probabilistic robustness guarantees.

For a given specification and dynamical system, model checking tools can be used to synthesize hybrid controllers based on a finite-state abstraction, which bisimulates the original continuous system \cite{fainekos2005b,tabuada2006,gol2015}. 
Mixed-integer programming was proposed more recently in \cite{karaman2008,wolff2014} for synthesizing trajectories that satisfy LTL specifications when the system dynamics are discrete-time linear or mixed logical dynamical (MLD) \cite{bemporad1999}. This approach takes advantage of mixed-integer optimization tools, instead of constructing a possibly large discrete abstraction.

Control strategies that ensure robust satisfaction of LTL specifications have been explored in the context of \emph{reactive planning}. In this framework, the uncertain environment behavior and the task are described by a joint specification. 
In \cite{kress-gazit2009} this is used to synthesize controllers that achieve the specified tasks for all possible environment behaviors. However, these behaviors are encoded purely in the specification and therefore disturbances that are correlated in time cannot be incorporated.
A more general description of the uncertainty is considered in \cite{frick2017b}, allowing for dynamic disturbances. A robust optimization problem is solved to obtain policies that ensure the satisfaction of the specification for all possible realizations of the uncertainties.
Signal temporal logic (STL), an alternative specification language that captures robustness, is used in \cite{raman2015}. where they iteratively construct a finite set of worst-case realizations and find robust trajectories for this set, however no robustness guarantees for the original problem are given.

Methods dealing with the probabilistic nature of uncertainties affecting the dynamical system or the specification have also been investigated. 
For discrete state spaces, uncertain Markov decision processes (MDPs) with general LTL specifications are considered in \cite{ding2011,wolff2012} with the goal of maximizing the probability of satisfying the specification.
MDPs are also used with a probabilistic specification language in \cite{lahijanian2012}.
In \cite{kamgarpour2013} a subset of LTL specifications is combined with stochastic hybrid systems, giving rise to a stochastic reachability problem. This is extended to include probabilistic uncertainties in the location of goal and obstacle sets in \cite{kamgarpour2017}.
However, these approaches require at least partial a-priori knowledge of the uncertainty distribution and do not consider performance criteria other than maximizing the probability of satisfying the specification.
Probabilistic STL specifications with Gaussian distributed uncertainties entering linearly in the atomic propositions are considered in \cite{sadigh2016}. The specification constraint is written as a chance constraint giving rise to a mixed-integer semi-definite program. Considering only Gaussan distributions enables to provide probabilistic guarantees for the satisfaction of the specification.
Similarly, in \cite{farahani2017} chance constraints are used to encode linear system dynamics with Gaussian distributed additive uncertainty and a restricted subset of STL specifications.
An approach based on sampled realizations of the uncertainty is proposed in \cite{farahani2015} for STL specifications, reducing the problem to a robust mixed-integer program. However no probabilistic guarantees are given.
Other works investigating data-driven approaches, have proposed methods for inference of specifications from samples \cite{kong2014}, instead of policy synthesis.

We propose a novel data-driven approach for synthesizing robust policies satisfying LTL specifications in the presence of uncertainties that enter nonlinearly into the known system dynamics and the specification. These policies are synthesized directly from samples of the uncertainty and no a-priori knowledge of the distribution or support of these uncertainties is required.
Our contributions are as follows:
\begin{enumerate*}[label=\roman*)]
	\item We present a data-driven method for robust planning under uncertainty.
	\item We synthesize \emph{disturbance feedback policies} that are functions of past observed uncertainties. The use of feedback mitigates the effect of the uncertainties by incorporating knowledge about past realizations in real-time. Compared to static plans, feedback policies ensure feasibility for a larger class of robust control problems \cite{goulart2006} and can improve control performance.
	This is in contrast to most other methods that only find open-loop policies based on knowledge of the uncertainty distribution or its support. The presented approach is therefore substantially more general than previous work.
	\item We provide a-priori probabilistic guarantees for the satisfaction of the specification, utilizing tools from scenario optimization \cite{esfahani2015}. The degree of guaranteed robustness grows with the number of samples used in the policy synthesis.
	\item The desired policies can be computed \emph{offline} as solutions of mixed-integer programs. These problems can be solved using off-the-shelf solvers such as CPLEX~\cite{cplex}.
	The policies are then implemented \emph{online} adapting to observed uncertainties in \emph{real-time}.
	We illustrate the performance of the method on two numerical case studies.
\end{enumerate*}

\paragraph{\textbf{Notation}} For a matrix $A \in \reals{n\times m}$, $[A]_{i,j}$ is the sub-matrix of block coordinates $i,j$.
Further, given $x \in \reals{n}$, $y \in \reals{m}$, we use $(x,y) = [\begin{matrix}x^\transp & y^\transp\end{matrix}]^\transp \in \reals{n+m}$ interchangeably. The $i$-th element of $x$ is $[x]_i = x_i$.
Given a random variable $w$, we say that samples $w^{(k)}$ of $w$ are i.i.d. if they are independent and identically distributed.
By $\uniform(a,b)$ we denote the continuous uniform distribution supported on the interval $[a,b]$, with $a,b\in\reals{}$ and $a \leq b$.
By $ \esssup_{x \in \set{X}} f(x)$ we denote the essential supremum of $f : \reals{n} \rightarrow \reals{}$ on $\set{X} \subseteq \reals{n}$.
We let $\vert \set{J} \vert$ denote the cardinality of a finite set $\set{J}$.
Given $x\in \reals{}$, $\ln(x)$ is the natural logarithm of $x$. Euler's number is denoted by $e$.

\section{Problem formulation} \label{sec:problem}
We consider uncertain discrete-time systems of the form:
\begin{equation} \label{eq:system_dynamics}
	x_{k+1} = A(w_k)x_k + B(w_k)u_k + c(w_k)\,,
\end{equation}
where $x_k \in \reals{n_x}$ is the system state and $w_k \in \reals{n_w}$ is the uncertain state of the disturbance at time $k$. The control input $u_k \in \reals{n_u}$ is applied between time $k$ and $k+1$. $A$, $B$ and $c$ are possibly nonlinear functions of appropriate dimensions. Note that \eqref{eq:system_dynamics} is \emph{affine parameter varying} and generalizes the class of systems considered in \cite{frick2017b}.
Furthermore, this work can be extend to include MLD systems \cite{bemporad1999}.

\subsection{Temporal logic specifications in uncertain environments}
A formula in LTL is a combination of \emph{atomic propositions} $p$ taken from a finite set $\apSet = \{p_1,\dots, p_m\}$, propositional logic operators $\neg (not)$, $\land (and)$, $\lor (or)$, and temporal operators $\lnext\:(next)$, $\until (until)$, $\release (release)$. We consider LTL formulae in positive normal form \cite{baier2008}, defined via the grammar:
\begin{equation*}
	p \mid \neg p \mid \phi \land \psi \mid \phi \lor \psi \mid \lnext \phi \mid \phi \until \psi \mid \phi \release \psi\,,
\end{equation*}
where $\phi$, $\psi$ are LTL formulae, and atomic propositions take values in $\{\true,\false\}$. Note that every LTL formula can be rewritten in positive normal form \cite{baier2008}. We consider bounded LTL formulae without loops \cite[Definition~2.1]{biere2006} that are affected by uncertainty. More specifically, we consider a set of atomic propositions $\apSet$ such that each atomic proposition $p_i\in \apSet$ is associated with a set defined over the state-disturbance space:
\begin{equation} \label{eq:polyhedral_set}
	\set{P}_i := \big\{ (x,w) \in \reals{n_x + n_w} \sep{\big} P_i(w) x \leq \rho_i(w) \big\}\,.
\end{equation}
The functions $P_i : \set{S} \rightarrow \reals{r_i \times n_x}$ and $\rho_i : \set{S} \rightarrow \reals{r_i}$ are defined over a common compact domain $\set{S} \subseteq \reals{n_w}$, and are nonlinear and finite-valued. The set $\set{S}$ can be inferred from samples, as noted in \refsec{section:scenario_approach}. We say that $(x,w)$ satisfies $p_i$, i.e., $(x,w) \models p_i$, if and only if $(x,w) \in \set{P}_i$. For fixed $w$, $\set{P}_i$ imposes polyhedral constraints on $x$ with $r_i$ inequalities.
The LTL semantics are then defined over the augmented state-disturbance state $(x,w)$, as in \cite[Equation~(2)]{frick2017b}.
As outlined in \refexa{exa:prop}, \eqref{eq:polyhedral_set} can be used to encode obstacles with uncertain position and orientation.
\begin{example} \label{exa:prop}
	Given state $x := (x_1,x_2) \in \reals{2}$ and uncertain obstacle position and orientation $w := (w_1,w_2,w_3) \in \reals{2}\times [0,2\pi)$. We consider a rectangular obstacle centered at $(w_1,w_2)$ and rotated by $w_3$. The following nonlinear inequalities describe the atomic proposition $p_{\rm obs}$ which is $\true$ if the state is inside the obstacle:
	\begin{equation*}
		\begin{bsmallmatrix} I \\ -I \end{bsmallmatrix}\mathcal{R}(w_3) \begin{bsmallmatrix}x_1\\x_2\end{bsmallmatrix} \leq \begin{bsmallmatrix} I \\ -I \end{bsmallmatrix}\mathcal{R}(w_3) \begin{bsmallmatrix}w_1\\w_2\end{bsmallmatrix} + \tfrac{1}{2}\begin{bsmallmatrix}b\\-b\end{bsmallmatrix}\,,
	\end{equation*}
	with side lengths of the rectangle $b := (b_1,b_2) \in \reals{2}$, identity matrix $I \in \reals{2\times 2}$ and rotation matrix $\mathcal{R}(w_3) := \big[\begin{smallmatrix} \cos(w_3) & -\sin(w_3) \\ \sin(w_3) & \cos(w_3) \end{smallmatrix}\big]$.
\end{example}

\subsection{Robust policy synthesis problem}
For system \eqref{eq:system_dynamics} and given a fixed \emph{planning horizon} $N$, we denote the state trajectory of length $N$ starting from $x_0$ as $\mathbf{x} :=  (x_0, x_1 , \ldots , x_N )$. The trajectory $\mathbf{x}$ is uniquely defined by its \emph{initial state} $x_0$, the input sequence $\mathbf{u} := (u_0, \ldots , u_{N-1})$ and a realization of the disturbance sequence $\mathbf{w} := (w_0, \ldots , w_N)$, where $\mathbf{w}$ is a random variable defined over the probability space $(\setbf{W}, \set{F}, \mathbb{P})$.
In this work, we consider the case where no direct knowledge of the distribution of $\mathbf{w}$ is available. Instead, we assume that a sufficient number of sampled realizations can be obtained, e.g. from historical data or from probabilistic models \cite{jordan1998}.
We furthermore do not make any assumptions about $\setbf{W}$ or more generally the distribution $\mathbb{P}$ of $\mathbf{w}$. For instance, $\setbf{W}$ may be non-convex and the disturbances may be correlated over time. 

In practical applications, past disturbances can be measured and then acted upon. Therefore, we consider \emph{disturbance feedback policies} of the form $\mathbf{u}(\mathbf{w}):= ( u_0(w_0), \ldots , u_{N-1}(w_0, \dots, w_{N-1}) )$. This allows us to synthesize \emph{causal} control policies that can react to disturbances in real-time based on observed uncertainty realizations. We indicate all such policies by $\mathbf{u}(\cdot) : \reals{Nn_w} \rightarrow \reals{Nn_u}$.

Given an LTL specification $\varphi$, our goal is to find a causal disturbance feedback policy $\mathbf{u}(\cdot)$ such that the resulting trajectory $\mathbf{x}$ satisfies the specification $\varphi$ robustly, i.e.,
\begin{equation*}
	(\mathbf{x} , \mathbf{w}) \models \varphi \quad \forall \mathbf{w} \in \setbf{W}\,,
\end{equation*}
where $(\mathbf{x} , \mathbf{w}) \models \varphi$ denotes the satisfaction of the formula $\varphi$.
Moreover, we want to minimize an objective function $l(\mathbf{u}(\cdot))$ with additional state and input constraints, yielding the problem: 
\begin{subequations} \label{eq:original_prob}
\begin{align}
	\min_{\mathbf{u}(\cdot) \text{ causal}} \: & l(\mathbf{u}(\cdot)) \\[-0.5em]
	\suchthat & \begin{gathered}\mathbf{x},\mathbf{u}(\mathbf{w}),\mathbf{w} \text{ satisfy system dynamics \eqref{eq:system_dynamics}} \\[-0.5em] \text{starting from initial state } x_0\,,\end{gathered} \\
	& (\mathbf{x}, \mathbf{w}) \models \varphi \hspace{1em} \hspace{1em} \forall \mathbf{w} \in \setbf{W}\,, \label{constraint_1} \\
	&(\mathbf{x},\mathbf{u}(\mathbf{w})) \in \setbf{X} \times  \setbf{U} \hspace{1em} \forall \mathbf{w} \in \setbf{W}\,, \label{eq:state_input_constr}
\end{align}
\end{subequations}
where $\setbf{X} := \set{X} \times \dots \times \set{X} \subseteq \reals{(N+1)n_x}$ and $\setbf{U} := \set{U} \times \dots \times \set{U} \subseteq \reals{Nn_u} $ are compact polyhedra.
The objective function $l(\mathbf{u}(\cdot))$ can e.g. encode a worst-case cost $l(\mathbf{u}(\cdot)) = \max_{\mathbf{w} \in \setbf{W}} \tilde{l}(\mathbf{u}(\mathbf{w}),\mathbf{x}, \mathbf{w})$ or an expected cost $l(\mathbf{u}(\cdot)) = \mathbb{E}_{\mathbf{w} \in \setbf{W}} [ \tilde{l}(\mathbf{u}(\mathbf{w}), \mathbf{x},\mathbf{w})]$.

Unfortunately, \refprob{eq:original_prob} cannot be solved since the support set $\setbf{W}$ is not known and only samples of $\mathbf{w}$ are available. This means that any guarantee is necessarily probabilistic. 
Moreover, \refprob{eq:original_prob} is a robust optimization problem over the infinite dimensional space of functions $\mathbf{u}(\cdot)$ and the specification constraint \eqref{constraint_1} is non-convex and not in a standard form recognized by off-the-shelf optimization solvers. 
We will deal with these challenges in the course of the next two sections.

\section{Reformulation as finite dimensional robust program} \label{sec:reformulation}
We will first show how the robust specification constraint~\eqref{constraint_1} of \refprob{eq:original_prob} can be brought into a more standard form, as a set of robust mixed-integer nonlinear constraints.
Then we address the infinite dimensionality of \refprob{eq:original_prob} by introducing a parameterization of $\mathbf{u}(\cdot)$. This allows us to optimize over a finite number of real-valued parameters, rather than over functions. 

We express the system dynamics \eqref{eq:system_dynamics} in terms of the initial state $x_0$, the uncertainty $\mathbf{w}$ and the policy $\mathbf{u}(\cdot)$:
\begin{equation} \label{eq:state_evolution}
	\mathbf{x} = \mathbf{A}(\mathbf{w}) x_0 + \mathbf{B}(\mathbf{w}) \mathbf{u}(\mathbf{w}) + \mathbf{c}(\mathbf{w})\,,
\end{equation}
where $\mathbf{A}(\mathbf{w}) \in \reals{(N+1)n_x \times n_x}$ is a block matrix with blocks $[\mathbf{A}(\mathbf{w})]_{i,1} := \prod_{j=0}^{i-2} A(w_{i-2-j}) \in \reals{n_x \times n_x}$ for $i=1,\ldots, N+1$, 
$\mathbf{B}(\mathbf{w})\in \reals{(N+1)n_x \times Nn_u}$ is a block matrix with blocks $[\mathbf{B}(\mathbf{w})]_{i,j} \in\reals{n_x \times n_u}$ such that for $i=1, \ldots,N+1$ and $j=1, \ldots,N$
\begin{equation*}
	[\mathbf{B}(\mathbf{w})]_{i,j} := \begin{cases} \mathbf{0} \in\reals{n_x \times n_u}&  \text{if }\: i\leq j\,, \\ \big(\prod_{k=2}^{i-j} A(w_{i-k})\big)B(w_{j-1}) & \text{otherwise}\,, \end{cases}
\end{equation*}
and a vector $\mathbf{c}(\mathbf{w}) \in \reals{(N+1)n_x}$ such that\\$ [\mathbf{c}(\mathbf{w})]_i = \sum_{k=0}^{i-2} \big(\prod_{j=0}^{i-k-3} A(w_{i-2-j})\big) c(w_k) \in \reals{n_x}$.

\subsection{Mixed-integer encoding of specification constraints} \label{sec:mipspec}
In \refprop{prop_equivalence}, we show that the specification constraint~\eqref{constraint_1} can be written as a set of mixed-integer nonlinear constraints. The proof employs the Big-M reformulation \cite{bemporad1999}, which relies on $M_i^{j+}$ and $M_i^{j-}$ being computable for $i=1,\ldots,|\apSet|$ and $j=1,\ldots,r_i$:
\begin{subequations} \label{eq:bigm}
\begin{align}
	\infty > M_i^{j+} &\geq \hspace*{-1.5em}\sup_{(x,w) \in \set{X} \times \set{S}} \hspace*{-1.3em} [P_i(w)]_j x - [\rho_i(w)]_j\,,\\
	-\infty < M_i^{j-} &\leq \hspace*{-1.5em}\inf_{(x,w) \in \set{X} \times \set{S}} \hspace*{-1.3em} [P_i(w)]_j x - [\rho_i(w)]_j\,.
\end{align}
\end{subequations}
\begin{proposition}\label{prop_equivalence}
	For given $\mathbf{x}$ and realization $\mathbf{w}$ the specification constraint $(\mathbf{x}, \mathbf{w}) \models \varphi$ can be equivalently represented as 
	\begin{equation} \label{eq:NLMI_constraints}
		\exists \delta \in \Delta : F^x(\mathbf{w}) \mathbf{x} + F^\delta(\mathbf{w}) \delta + f(\mathbf{w}) \leq 0\,,
	\end{equation}
	for appropriate functions $F^x$, $F^\delta$ and $f$, by introducing the auxiliary variables $\delta \in \Delta := \reals{n_c} \times \binaries{n_b}$.
\end{proposition}
\begin{proof}
The proof follows analogously to \cite{wolff2014}. However, in~\cite{wolff2014} the encoding of the atomic propositions is linear, whereas in this work it is a nonlinear function of the uncertainty $w$.
For each atomic proposition $p_i \in \apSet$, we introduce auxiliary variables $\delta_i \in \binaries{r_i}$ enforcing that $[\delta_i]_j = 1$ if and only if $[P_i(w)]_j x \leq [\rho_i(w)]_j$. Using the Big-M formulation \cite{bemporad1999}, this is encoded by
\begin{align*}
	& [P_i(w)]_j x \leq [\rho_i(w)]_j + M_i^{j+} (1- [\delta_i]_j)\,, \\ 
	& [P_i(w)]_j x > [\rho_i(w)]_j + M_i^{j-} [\delta_i]_j\,,
\end{align*}
for $j=1,\ldots,r_i$, with $M_i^{j+}$ and $ M_i^{j-}$ as in \eqref{eq:bigm}.
As shown in \cite{wolff2014}, the satisfaction of $p_i$ can be encoded via \emph{linear} constraints on $\delta_i$.
Moreover, any finite Boolean combination of atomic propositions can be encoded with \emph{linear} constraints involving the binaries $\delta_i$ and additionally introduced continuous auxiliary variables, e.g., for a formula $\varphi = p_1 \lor p_2$ we introduce an additional continuous auxiliary variable $\delta_\varphi \in [0,1]$ such that $\delta_1,\delta_2 \leq \delta_\varphi$ and $\delta_\varphi \leq \delta_1 + \delta_2$. Note that $\delta_\varphi$ naturally takes only values in $\binaries{}$, due to $\delta_1,\delta_2 \in \binaries{}$.
Any bounded-time LTL formula can be written as such a finite Boolean combination of atomic propositions \cite{wolff2014}. Therefore, $F^x$, $F^\delta$ and $f$ can be constructed from all the aforementioned constraints and the dimensions of $\Delta$ corresponds to the total number of introduced auxiliary continuous and binary variables.
\end{proof}

\begin{subequations} \label{eq:condensedconstraint}
We use the state equation \eqref{eq:state_evolution} to express \eqref{eq:NLMI_constraints} as a function of only $x_0$, $\mathbf{u}(\cdot)$, $\delta$ and $\mathbf{w}$, obtaining the following constraint
\begin{equation} 
	\exists \delta \in \Delta : F^x(\mathbf{w})\big(\mathbf{A}(\mathbf{w}) x_0 +  \mathbf{B}(\mathbf{w})\mathbf{u}(\mathbf{w})\big) +F^\delta(\mathbf{w}) \delta + g(\mathbf{w}) \leq 0 \,,
\end{equation}
where $g(\mathbf{w}) := F^x(\mathbf{w})\mathbf{c}(\mathbf{w}) + f(\mathbf{w})$. The polyhedral state-input constraints \eqref{eq:state_input_constr} can be expressed similarly as
\begin{equation}
	S^x(\mathbf{w})\big(\mathbf{A}(\mathbf{w}) x_0 +  \mathbf{B}(\mathbf{w})\mathbf{u}(\mathbf{w})\big) +  S^u\mathbf{u}(\mathbf{w}) + s(\mathbf{w})\leq 0\,,
\end{equation}
for appropriate functions $S^x$, $s$ and matrix $S^u$.
\end{subequations}

\subsection{Parametrized feedback policies}
Problem \eqref{eq:original_prob} is an optimization problem over the infinite space of functions $\mathbf{u}(\cdot)$.
To tackle the infinite dimensionality, the problem can also be understood in the context of multi-stage robust optimization \cite{bental2009}, where $\mathbf{u}(\mathbf{w})$ depends on the realization of $\mathbf{w}$. In this context, policies $\mathbf{u}(\cdot)$ are typically parametrized, see \cite{bertsimas2016}. Hence, we consider feedback policies of the form
\begin{equation} \label{eq:parametrized policies}
	\mathbf{u}(\mathbf{w}):= \mathbf{H} \slimcdot \kappa(\mathbf{w})\,,
\end{equation} 
where $\mathbf{H}\in \reals{Nn_u \times n_\kappa}$ is a matrix of parameters and the function $\kappa: \reals{(N+1)n_w} \rightarrow \reals{n_\kappa}$ is fixed a-priori. We assume that appropriate restrictions on $\mathbf{H}$ and $\kappa(\cdot)$ ensure the causality of $\mathbf{u}(\cdot)$ and, for $\mathbf{H}$, these restrictions are captured by the polyhedral set $\set{H} \subseteq \reals{Nn_u \times n_\kappa} $. We let $d \leq Nn_u n_\kappa$ be the number of free parameters of $\mathbf{H} \in \set{H}$. Note that $\set{H}$ can also include constraints that enforce a desired sparsity structure of the policy, e.g. fixing some entries to zero. 
Policies parametrized as \eqref{eq:parametrized policies} also include \emph{piecewise-affine policies}, a large and commonly used class of policies, as shown in \refexa{example:piecewise_affine}.
\begin{example}[Piecewise-affine policies]\label{example:piecewise_affine}
	A causal piecewise-affine policy with $P \in \naturals{}$ pieces can be written in form \eqref{eq:parametrized policies} as
	 \begin{equation*}
	 	\mathbf{H} \slimcdot \kappa(\mathbf{w}) = \begin{bmatrix} \mathbf{H}_1 & \cdots & \mathbf{H}_P \end{bmatrix} \slimcdot \begin{bmatrix} \kappa_1(\mathbf{w})^\transp & \dots & \kappa_P(\mathbf{w})^\transp \end{bmatrix}^\transp\,,
	 \end{equation*}
	 with $\mathbf{H}_i \in \reals{Nn_u \times (N+1)n_w+1}$ block lower triangular and
	 \begin{equation*}
	 	\kappa_i(\mathbf{w}) := \begin{cases}\begin{bsmallmatrix}1 \\ \mathbf{w}\end{bsmallmatrix} & \text{if } \mathbf{w} \in \setbf{K}_i\,, \\
	 	0 & \text{otherwise}\,, \end{cases}
	 \end{equation*}
	 for $i=1,\ldots,P$, where $\{\setbf{K}_i\}_{i=1}^P$ is a partition of $\reals{(N+1)n_w}$.
\end{example}

Restricting \refprob{eq:original_prob} to parametrized policies~\eqref{eq:parametrized policies} yields a finite dimensional \emph{inner approximation}, the robust program $\RP$.
\begin{equation*}
	\RP \left\{ \begin{aligned}
		\min_{\mathbf{H}\in \set{H}} \: & {J(\mathbf{H})}  \nonumber \\
		\suchthat &  \exists \delta \in \Delta : \gamma_\varphi(\mathbf{H}, \delta ,\mathbf{w}) \leq 0 \quad \forall \mathbf{w} \in \setbf{W} \,, \\
		& \gamma_s(\mathbf{H},\mathbf{w}) \leq 0 \quad \forall \mathbf{w} \in \setbf{W}\,,
	\end{aligned} \right.
\end{equation*}
where using \eqref{eq:condensedconstraint} and \eqref{eq:parametrized policies} we have defined
\begin{align*}
	\gamma_\varphi(\mathbf{H}, \delta ,\mathbf{w}) &:= F^x \big(\mathbf{A}(\mathbf{w}) x_0 + \mathbf{B}(\mathbf{w})\mathbf{H}\slimcdot \kappa(\mathbf{w})\big) + F^\delta(\mathbf{w}) \delta + g(\mathbf{w})\,, \\
	\gamma_s(\mathbf{H} ,\mathbf{w}) &:= S^x\big(\mathbf{A}(\mathbf{w}) x_0 + \mathbf{B}(\mathbf{w})\mathbf{H}\slimcdot \kappa(\mathbf{w})\big) + S^u \mathbf{H}\slimcdot \kappa(\mathbf{w}) + s(\mathbf{w})\,.
\end{align*}
We express the specification constraint of $\RP$ more compactly as
\begin{equation} \label{eq:specconst}
	\eta_\varphi(\mathbf{H},\mathbf{w}) := \min_{\delta \in \Delta} \max_{j} \: [\gamma_\varphi(\mathbf{H}, \delta ,\mathbf{w})]_j \,.
\end{equation}

Standard techniques for solving robust programs replace the robust constraints with their dualizations \cite{goulart2006}. The resulting problems can then be solved using off-the-shelf solvers. However, these reformulations rely on strong duality which in general does not hold  since the specification constraint function $\eta_\varphi(\mathbf{H},\mathbf{w})$ is piecewise-affine non-convex in $\mathbf{H}$ due to the non-convexity of $\Delta$, and non-concave in $\mathbf{w}$ \cite{frick2017b}. 
Finally, we only have access to samples $\mathbf{w}^{(k)}$ of $\mathbf{w}$ and in particular $\setbf{W}$ is not known directly. As a result $\RP$ is still very challenging to solve.

In the next section we show how a set of such samples, or scenarios, can be used to obtain an approximation of $\RP$ via the \emph{scenario approach} \cite{campi2008}. The resulting problem can then be solved using off-the-shelf solvers. Naturally the degree of approximation and the obtainable guarantees will be probabilistic. 

\section{Generalization guarantees from samples via the scenario approach} \label{section:scenario_approach}
We assume that a collection $\setbf{W}^{K} := \{\mathbf{w}^{(1)}, \ldots, \mathbf{w}^{(K)} \} \subseteq \setbf{W}$ of $K$ i.i.d. samples of $\mathbf{w}$ is given and that we have no direct knowledge about the distribution $\mathbb{P}$ of $\mathbf{w}$ or its support $\setbf{W}$. Therefore, $\RP$ cannot be solved directly. Instead, we construct a \emph{sampled} optimization problem based on $\RP$ that utilizes the \emph{multisample} $\setbf{W}^{K}$. For solutions of this sampled problem, we derive a generalization guarantee on the constraint satisfaction of $\RP$, which will allow us to relate this solution to the robust program $\RP$, in a probabilistic sense.

Given a multisample $\setbf{W}^{K}$ that follows the distribution $\mathbb{P}^K$, we consider the \emph{scenario} version  \cite{campi2008,calafiore2010,schildbach2013} of $\RP$. 
We split the multisample $\setbf{W}^{K}$ into two groups, the first $K_\varphi \in \naturals{}$ samples for the specification constraint $\eta_\varphi$ and the remaining $K_s := K-K_\varphi$ samples for the state-input constraint $\gamma_s$. This allows weighing the relative importance of the two constraints and obtaining individual probabilistic guarantees that emphasize one constraint over the other. Note that the two constraints could be split further, giving more fine-grained control over the probabilistic guarantees.
For a given $\mathbb{K} := (K_\varphi, K_s)$ the scenario version of $\RP$ is
\begin{subequations}
\begin{equation*} 
	\SP{\mathbb{K}}\left\{ \begin{aligned}
		\min_{\mathbf{H}\in \set{H}} \: & {J(\mathbf{H})}\\
		\suchthat & \eta_\varphi(\mathbf{H}, \mathbf{w}) \leq 0 \quad \forall \mathbf{w} \in \{\mathbf{w}^{(1)},\ldots,\mathbf{w}^{(K_\varphi)}\}\,, \\
		& \gamma_s(\mathbf{H}, \mathbf{w}) \leq 0 \quad \forall \mathbf{w} \in \{\mathbf{w}^{(K_\varphi+1)},\ldots,\mathbf{w}^{(K)}\}\,.
	\end{aligned} \right.
\end{equation*}
\end{subequations}
For the remainder of this paper, we assume $J(\mathbf{H})$ to be linear or convex quadratic in $\mathbf{H}$. This can include expected value objectives approximated using sample average approximation \cite{shapiro2013}, or worst-case objectives if $J$ is linear in $\mathbf{H}$ for fixed $\mathbf{w}$. Furthermore, we make the standard assumption that $\SP{\mathbb{K}}$ is feasible for almost all realizations $\setbf{W}^K$ and that it admits a unique optimal solution \cite{campi2008}.
Note that existence and uniqueness can be relaxed via refinement techniques and using a suitable tie-breaking rule, respectively, see \cite{calafiore2010}.
We let $\SPopt{\mathbb{K}}$ be the optimizer of $\SP{\mathbb{K}}$. Note that $\SPopt{\mathbb{K}}$ is itself a random variable, since it depends on the multisample $\setbf{W}^{K}$.

We use the \emph{scenario approach} to establish a probabilistic relationship between $\SP{\mathbb{K}}$ and $\RP$. A basic concept used in the scenario approach is the \emph{violation probability} $V_\varphi(\mathbf{H})$ of constraint $\eta_\varphi$ for a given $\mathbf{H}\in \set{H}$. It is the probability with which $\mathbf{H}$ violates the constraint function $\eta_\varphi(\cdot, \mathbf{w})$ for an arbitrary realization of $\mathbf{w}$. 
\begin{definition}[\cite{schildbach2013}, Violation probability]
	The violation probability of $\eta_\varphi$ for a given $\mathbf{H}\in \set{H}$ is defined as
	\begin{equation*}
		V_\varphi(\mathbf{H}):=\mathbb{P} \{\mathbf{w} \in \setbf{W} \sep{} \exists j : [\eta_\varphi(\mathbf{H},\mathbf{w})>0]_j\}\,.
	\end{equation*} 
\end{definition}
\noindent The violation probability of $\gamma_s$ is defined analogously.
We would like to guarantee that the violation probability of the solution $\SPopt{\mathbb{K}}$ is smaller than a given $\epsilon_\varphi \in (0,1)$. However, $V_\varphi(\SPopt{\mathbb{K}})$ is itself probabilistic, since $\SPopt{\mathbb{K}}$ is a random variable. Therefore, we can only guarantee $V_\varphi(\SPopt{\mathbb{K}}) \leq \epsilon_\varphi$ with a certain confidence.
The generalization results from the scenario approach literature provide a bound $(1-\beta_\varphi) \in (0,1)$ on the confidence, for a given $\mathbb{K}$ of $K$ and an allowable violation parameter $\epsilon_\varphi$. More formally they guarantee
\begin{equation} \label{eq:guarantee}
	\mathbb{P}^K[V_\varphi(\SPopt{\mathbb{K}}) > \epsilon_\varphi] \leq \beta_\varphi\,.
\end{equation}
Alternatively, the results can be used to lower bound the number of samples $K$ needed to ensure a desired constraint violation with a given confidence. We equivalently define $\epsilon_s,\beta_s \in (0,1)$ for $\gamma_s$.

In many practical applications it is difficult or expensive to obtain samples of $\mathbf{w}$. Furthermore, the number of constraints, and therefore the complexity of $\SP{\mathbb{K}}$ grows with the number of samples. We thus want to select the smallest number of samples $K$ such that the given violation parameters $\epsilon_\varphi$, $\epsilon_s$ and confidence parameters $\beta_\varphi$, $\beta_s$ are still ensured.
For the convex case, i.e., when the objective function $J$ and constraint functions $\eta_\varphi$, $\gamma_s$ are convex for fixed $\mathbf{w}$, bounds on $K$ can be obtained which depend on the \emph{support dimension} of $\eta_\varphi$ and $\gamma_s$ \cite{schildbach2013}. To define the support dimension, we first define a \emph{support constraint}:
\begin{definition}[\cite{campi2008}, Support constraint]
	A constraint of $\SP{\mathbb{K}}$ is a support constraint if its removal strictly improves the optimal cost of $\SP{\mathbb{K}}$. 
\end{definition}
\begin{definition}[\cite{schildbach2013}, Support dimension] \label{def:supportdim}
	The support dimension of $\eta_\varphi$ is the smallest $\zeta_\varphi \in \naturals{}$ s.t. $\esssup_{\setbf{W}^K} \{ \,| sc_\varphi(\SP{\mathbb{K}}) |\, \} \leq \zeta_\varphi$, where $sc_\varphi(\SP{\mathbb{K}})$ is the set of support constraints contributed by the sampled instances of $\eta_\varphi$.
\end{definition}
\noindent The support dimension $\zeta_s$ of $\gamma_s$ is defined analogously. 
For the convex case defined above, the confidence can be bounded as a function of $K_\varphi$, $\epsilon_\varphi$ and $\zeta_\varphi$ \cite{schildbach2013}, i.e.,
\begin{equation*}
	\mathbb{P}^K[V_\varphi(\SPopt{\mathbb{K}})> \epsilon_\varphi] \leq \sum_{i=0}^{\zeta_\varphi-1} \binom{K_\varphi}{i} \epsilon_\varphi^i (1-\epsilon_\varphi)^{K_\varphi-i}\,.
\end{equation*}
When the support dimension $\zeta_\varphi$ is known, this can be used to select the number of samples $K_\varphi$ such that \eqref{eq:guarantee} is ensured for a given $\epsilon_\varphi$ and $\beta_\varphi$. However, computing the support dimension can be challenging \cite{zhang2015}. In the convex case it can be upper bounded by the total number of optimization variables, using Helly's theorem \cite{campi2008,calafiore2010}. In \cite{schildbach2013} it was shown that $\zeta_\varphi$ is upper bounded by the so called \emph{support rank} of the constraint $\eta_\varphi$, which in some cases can be significantly smaller than the number of optimization variables.

Recall the robust program $\RP$ considered in this paper. 
The constraint $\gamma_s$ is convex in $\mathbf{H}$. Therefore, its support dimension $\zeta_s$ can be upper bounded by $d$, the number of free parameters of the policy.
As a consequence, $\epsilon_s, \beta_s$ and $\zeta_s$ can be used to determine $K_s$ such that \eqref{eq:guarantee} is satisfied for the state-input constraint $\gamma_s$.
However, the specification constraint $\eta_\varphi$ is non-convex and constraints of form \eqref{eq:specconst} have not been considered previously in the literature.
In fact, $\RP$ does not fall into any of the classes of robust non-convex programs considered in \cite{calafiore2012,grammatico2016}.
In \refprop{prop:infinite_helly_dimension} we show that the support dimension $\zeta_\varphi$ cannot be bounded a-priori, and therefore arguments as in \cite{schildbach2013,calafiore2012,grammatico2016} cannot be used to give bounds on the number of samples $K_\varphi$ that enforce guarantees on $V_\varphi(\SPopt{\mathbb{K}})$.
\begin{proposition} \label{prop:infinite_helly_dimension}
	There does not exist an upper bound for the support dimension $\zeta_\varphi$ of $\eta_\varphi$, independent from $K_\varphi$.
\end{proposition}
\vspace{-\baselineskip}
\begin{proof} 
	See \refapp{app:proofs}, page~\pageref{app:proofs}.
\end{proof}
This shows that the usual approaches for obtaining sampling bounds will not work for $\RP$.
In the following section, we propose appropriate inner approximations of $\RP$ that fall into the class of problems considered in \cite{frick2018}. We then show that we can still obtain guarantees on the violation probabilities $V_\varphi$ and $V_s$.

\subsection{Inner approximations with probabilistic guarantees}
Consider the robust program $\RP$. For a given $\mathbf{H} \in \set{H}$ and $\mathbf{w} \in \setbf{W}$, the existence of a $\delta \in \Delta$ such that $\gamma_\varphi(\mathbf{H}, \delta ,\mathbf{w}) \leq 0$ ensures that the specification $\varphi$ is satisfied. Hence, for a given $\mathbf{H}$, we can think of $\delta$ as a \emph{recourse decision} on $\mathbf{w}$, i.e., the decision $\delta$ can depend on $\mathbf{w}$. Such recourse variables are dealt with in the context of multistage adaptive optimization, by parameterizing the search space of $\delta$. This leads to an inner approximation of $\RP$. Mixed-integer recourse variables, as considered in this work, can be parametrized as \emph{piecewise-constant} over $\setbf{W}$ \cite{bertsimas2016,bertsimas2017}, i.e.,
\begin{equation}\label{eq:piecewise_delta}
	\delta(\mathbf{w}):= \delta_i \text{ if } \mathbf{w} \in \setbf{D}_i \quad \text{for } i=1,\ldots,P_\delta\,,
\end{equation}
where $\{\setbf{D}_i\}_{i=1}^{P_\delta}$ is a partition of $\reals{(N+1)n_w}$, which can e.g. be obtained following \cite{postek2016}. Note that $\delta$ is only used to encode the specification, hence its parametrization can be non-causal.
Moreover, the fact that the parametrization is piecewise-constant is not conservative in this case, since the continuous components of $\delta$ must take values in $\binaries{}$ by construction in \refsec{sec:mipspec}.
The parametrization \eqref{eq:piecewise_delta} of $\delta$ leads to the following inner approximation of $\RP$: 
\begin{equation*} 
	\RPapprox\left\{ \begin{aligned}
		\min_{\mathbf{H} \in \set{H},\, \greekbf{\delta}\in \greekbf{\Delta}} \: &  {J(\mathbf{H})} \\
		\suchthat & \gamma_\varphi(\mathbf{H},\delta_{\set{I}(\mathbf{w})},\mathbf{w}) \leq 0 \quad \forall \mathbf{w} \in \setbf{W}\,, \\
		& \gamma_s(\mathbf{H},\mathbf{w}) \leq 0 \quad \forall \mathbf{w} \in \setbf{W}\,,
	\end{aligned} \right. 
\end{equation*}
where $\greekbf{\delta}:= (\delta_1, \ldots , \delta_{P_\delta}) \in \greekbf{\Delta}:=\Delta \times \cdots \times\Delta$ is the stacked version of copies of $\delta$ for each element of the partition, and $\set{I}(\mathbf{w}) := \{ i \in \{1,\dots, P_\delta\} \sep{} \text{s.t. } \mathbf{w} \in \setbf{D}_i \}$.

Given a multisample $\setbf{W}^K$, the scenario version of $\RPapprox$ is
\begin{equation*} 
	\SPapprox{\mathbb{K}} \left\{ \begin{aligned}
		\min_{\mathbf{H}\in \set{H},\, \greekbf{\delta}\in \greekbf{\Delta}} &  {J(\mathbf{H})} \\
		\suchthat & \gamma_\varphi(\mathbf{H}, \delta_{\set{I}(\mathbf{w})}, \mathbf{w}) \leq 0 \quad \forall \mathbf{w} \in \{\mathbf{w}^{(1)},\ldots,\mathbf{w}^{(K_\varphi)}\}\,, \\
		& \gamma_s(\mathbf{H}, \mathbf{w}) \leq 0 \quad \forall \mathbf{w} \in \{\mathbf{w}^{(K_\varphi+1)},\ldots,\mathbf{w}^{(K)}\}\,,
	\end{aligned} \right.
\end{equation*}
where the optimizer $(\SPapproxoptH{\mathbb{K}}, \SPapproxoptd{\mathbb{K}})$ is assumed to exist and be unique.
Note that the domain $\set{S}$ in \eqref{eq:polyhedral_set} can be set to $\set{S} := \{\mathbf{w}^{(1)},\ldots,\mathbf{w}^{K_\varphi}\}$, without loss of generality.
Let $\widetilde{V}_\varphi, \widetilde{V}_s $ be the violation probabilities related to the two constraints in $\RPapprox$. That is, for given $\mathbf{H},\greekbf{\delta}$ we have
\begin{align*}
	\widetilde{V}_\varphi(\mathbf{H}, \greekbf{\delta}) &:= \mathbb{P}\{ \mathbf{w}\in \setbf{W} \sep{} \exists j :[\gamma_\varphi(\mathbf{H}, \delta_{\set{I}(\mathbf{w})}, \mathbf{w})]_j > 0\}\,, \\
  \widetilde{V}_s(\mathbf{H}) &:= \mathbb{P}\{ \mathbf{w}\in \setbf{W} \sep{} \exists j :[\gamma_s(\mathbf{H}, \mathbf{w})]_j > 0\}\,.
\end{align*}
The violation probabilities related to $\RP$ and $\RPapprox$ can be linked via the following lemma.
\begin{lemma} \label{lem:violation_probs}
	Consider, $\mathbb{K} = (K_\varphi,K_s)$ and $\epsilon_\varphi,\epsilon_s \in (0,1)$. Solutions $(\SPapproxoptH{\mathbb{K}},\SPapproxoptd{\mathbb{K}})$ of $\:\SPapprox{\mathbb{K}}$ satisfy
\begin{align*} 
	\mathbb{P}^K[V_\varphi(\SPapproxoptH{\mathbb{K}}) > \epsilon_\varphi] &\leq \mathbb{P}^K[\widetilde{V}_\varphi(\SPapproxoptH{\mathbb{K}}, \SPapproxoptd{\mathbb{K}}) > \epsilon_\varphi ]\,,\\
	\mathbb{P}^K[V_s(\SPapproxoptH{\mathbb{K}}) > \epsilon_s] &= \mathbb{P}^K[\widetilde{V}_s(\SPapproxoptH{\mathbb{K}}, \SPapproxoptd{\mathbb{K}}) > \epsilon_s ]\,.
\end{align*}
\end{lemma}
\begin{proof}
	By definition $\eta_\varphi(\SPapproxoptH{\mathbb{K}}, \mathbf{w}) \leq \gamma_\varphi(\SPapproxoptH{\mathbb{K}}, \delta_{\set{I}(\mathbf{w})}^\star[\mathbb{K}], \mathbf{w})$ for any $\mathbf{w} \in \setbf{W}$ and therefore $V_\varphi(\SPapproxoptH{\mathbb{K}}) \leq \widetilde{V}_\varphi(\SPapproxoptH{\mathbb{K}}, \SPapproxoptd{\mathbb{K}})$. Furthermore, $V_s$ and $\widetilde{V}_s$ refer to the exact same constraint $\gamma_s$, therefore $V_s(\SPapproxoptH{\mathbb{K}}) = \widetilde{V}_s(\SPapproxoptH{\mathbb{K}})$. This yields the result.
\end{proof}

\reflem{lem:violation_probs} allows us to obtain probabilistic guarantees for the solution $\SPapproxoptH{\mathbb{K}}$ of $\SPapprox{\mathbb{K}}$, relating it to the original robust program $\RP$.
The following main theorem gives bounds on the number of samples needed to obtain small constraint violation probabilities with a certain confidence.
\begin{theorem}\label{thm:1}
Let $d$ be the number of free parameters of the policy and $n_c$, $n_b$ the number of continuous and binary variables of $\delta$, respectively. Let $P_\delta$ be the number of elements of the partition for $\delta$. Let $\epsilon_\varphi, \epsilon_s \in (0,1)$ be the desired violation parameters and $\beta_\varphi, \beta_s \in (0,1)$ the desired confidence parameters. For $\mathbb{K} = (K_\varphi, K_s)$ such that
\begin{subequations} \label{eq:numsamp}
\begin{align}
	K_\varphi &\geq \frac{e}{e-1}\frac{1}{\epsilon_\varphi}\big(\ln{\left(\frac{2^{P_\delta n_b}}{\beta_\varphi}\right)} + d + P_\delta n_c - 1 \big)\,,\label{eq:numsamp:phi}\\
	K_s &\geq \frac{e}{e-1}\frac{1}{\epsilon_s}\big(\ln{\left(\frac{2^{P_\delta n_b}}{\beta_s}\right)} + d - 1 \big)\,,
\end{align}
\end{subequations}
the solution $\SPapproxoptH{\mathbb{K}}$ of $\:\SPapprox{\mathbb{K}}$ satisfies
\begin{equation*}
	\mathbb{P}^K[V_\varphi(\SPapproxoptH{\mathbb{K}}) > \epsilon_\varphi] \leq \beta_\varphi \text{ and }
	\mathbb{P}^K[V_s(\SPapproxoptH{\mathbb{K}}) > \epsilon_s] \leq \beta_s\,.
\end{equation*}
\end{theorem}
\begin{proof}
For each fixed binary configuration $\greekbf{\delta}$, $\RPapprox$ is a robust convex program and falls in the class of mixed-integer robust programs considered in \cite{frick2018}.
This is because, for fixed $\mathbf{w}$, the system dynamics are linear, the feedback policy is linear, and the atomic propositions are described by polyhedral constraints, which implies that $\gamma_\varphi$ is affine in $(\mathbf{H}, \delta_{\set{I}(\mathbf{w})})$ for a given $\mathbf{w} \in \setbf{W}$.
Therefore, applying \cite[Theorem~1 and Fact~1]{frick2018} we have
\begin{equation*}
	\mathbb{P}^K[\widetilde{V}_\varphi(\SPapproxoptH{\mathbb{K}}, \SPapproxoptd{\mathbb{K}}) > \epsilon_\varphi ]  \leq 2^{P_\delta n_b} \hspace*{-1.2em} \sum_{i=0}^{d+ P_\delta n_c-1} \hspace*{-0.2em} \binom{K_\varphi}{i} \epsilon_\varphi^i (1-\epsilon_\varphi)^{K_\varphi-i}\,.
\end{equation*}
Furthermore, from \cite[Equation~(3)]{frick2018} we obtain the bound \eqref{eq:numsamp:phi} for $K_\varphi$ in order to satisfy $\mathbb{P}^K[\widetilde{V}_\varphi(\SPapproxoptH{\mathbb{K}}, \SPapproxoptd{\mathbb{K}}) > \epsilon_\varphi ] \leq \beta_\varphi$. The same holds for $\widetilde{V}_s(\SPapproxoptH{\mathbb{K}})$, $K_s$, $\epsilon_s$ and $\beta_s$ with $d$ in place of $d + P_\delta n_c$, where $d + P_\delta n_c$ and $d$ are the total number of continuous variables present in constraints $\gamma_\varphi$ and $\gamma_s$, respectively, and thus upper bound their support ranks.  
The result then follows from \reflem{lem:violation_probs}.
\end{proof}

\begin{remark}\label{rem:tighter_bounds}
The bounds \eqref{eq:numsamp} are only sufficient. Tighter bounds can be obtained by performing a binary search on an implicit criterion outlined in \cite[Theorem~1]{frick2018}.
Moreover, $2^{P_\delta n_b}$ can be replaced by the number of feasible binary configurations, see \cite[Section~2.1]{frick2018}.
\end{remark}

\begin{remark}
	The number of constraints of $\SPapprox{\mathbb{K}}$ grows linearly with the number of decision variables, which depends on the parametrization of the policy, the size and encoding of the specification, see \cite{wolff2014}. Once $\SPapprox{\mathbb{K}}$ has been computed offline, the policy can be evaluated online with negligible computational effort. 
\end{remark}

\begin{algorithm}[h]
	\caption{Policy synthesis, general case} \label{alg:synthesis}
  	\begin{algorithmic}[1]
   		\Require{Policy parametrization $\mathbf{H} \slimcdot \kappa(\mathbf{w})$, parametrization of $\delta$}
   		\Require{Parameters $\epsilon_\varphi, \epsilon_s, \beta_\varphi, \beta_s \in (0,1)$}
   		\State $\mathbb{K} \gets (K_\varphi,K_s)$ according to \refthm{thm:1}
   		\State obtain multisample $\setbf{W}^{K_\varphi+K_s}$ \Comment{e.g. from historical data}
   		\State $(\SPapproxoptH{\mathbb{K}}, \SPapproxoptd{\mathbb{K}}) \gets$ solve $\SPapprox{\mathbb{K}}$ using $\setbf{W}^{K_\varphi+K_s}$
   		\State \Return $\SPapproxoptH{\mathbb{K}}$
  	\end{algorithmic}
\end{algorithm}

We outline the synthesis procedure in \refalg{alg:synthesis}. The solution enjoys the desired probabilistic guarantees provided in \refthm{thm:1}.

\begin{remark}
	MLD systems can be accommodated by parameterizing the discrete variables with causal piece-wise constant functions of the uncertainty. This is similar to how $\delta$ is parameterized, except for the causality of the policy.
\end{remark}
 
\subsection{Linear dependence on the uncertainty w}\label{subsec:linear_in_w}
We now consider the special case, where the dependence on the uncertainty $\mathbf{w}$ is linear.
Instead of solving the scenario program $\SPapprox{\mathbb{K}}$, we use samples of $\mathbf{w}$ to construct an approximation of the support set $\setbf{W}$ and solve a mixed-integer robust optimization problem. The resulting optimization problem can have much fewer constraints providing significant computational advantages over \refalg{alg:synthesis}, while providing similar probabilistic guarantees.

We consider discrete-time dynamics as in \eqref{eq:system_dynamics} with $A$ and $B$  constant, and $c(w)$ affine in $w$. For each atomic proposition $p_i \in \apSet$, $P_i$ is constant and $\rho_i(w)$ is affine in $w$. We only consider piecewise-affine policies, as introduced in \refexa{example:piecewise_affine}. 
As a result, the constraint functions $\gamma_\varphi$, $\gamma_s$ are piecewise-affine in $\mathbf{w}$, with affine pieces $\gamma_{\varphi,i}(\mathbf{H}_i, \delta ,\mathbf{w})$ and $\gamma_{s,i}(\mathbf{H}_i, \mathbf{w})$ defined over the partition $\{\setbf{K}_i\}_{i=1}^P$.

The support set $\setbf{W}$ of $\mathbf{w}$ can be estimated from an i.i.d. multisample $\setbf{W}^{K_w}$ using the scenario approach \cite{margellos2014}, where $K_w$ is the number of samples. We denote by $\Wapprox{K_w} := \{ \mathbf{w} \in \reals{(N+1)n_w} \sep{} \mathbf{W}\mathbf{w} \leq \mathbf{v} \}$, the polyhedral estimate of $\setbf{W}$ obtained from $\setbf{W}^{K_w}$, with $\mathbf{W} \in \reals{m \times (N+1)n_w}$ and $\mathbf{v} \in \reals{m}$. The estimate $\Wapprox{K_w}$ is computed by finding the parameters $\mathbf{W}$, $\mathbf{v}$ such that $\Wapprox{K_w}$ is small, in some appropriate sense, and such that all samples of $\setbf{W}^{K_w}$ are contained in $\Wapprox{K_w}$. When the matrix $\mathbf{W}$ is fixed, this reduces to solving a linear program. Note that \cite{margellos2014} generalizes to the case where the estimate $\Wapprox{K_w}$ is a finite union of polyhedra.

We assume that the disturbance feedback policy $\mathbf{u}(\cdot)$ is defined over a polyhedral partition $\{\setbf{K}_i\}_{i=1}^P$ of $\Wapprox{K_w}$ and, for simplicity, $\delta(\mathbf{w})$ is parametrized over the same partition. Then, $\RPapprox$ can be solved approximately via the following robust optimization problem
\begin{equation*} 
	\RPlin{K_w}\left\{ \begin{aligned}
		\min_{\mathbf{H} \in \set{H},\, \greekbf{\delta}\in \greekbf{\Delta}} & \:\, J(\mathbf{H}) \\[-1em]
	\suchthat & \hspace*{-0.5em}\max_{i=1,\ldots,P} \max_{\mathbf{w} \in \setbf{K}_i} \max_j \begin{bmatrix}\gamma_{\varphi,i}(\mathbf{H}_i, \delta_i, \mathbf{w}) \\ \gamma_{s,i}(\mathbf{H}_i, \mathbf{w})\end{bmatrix}_j \leq 0\,.
	\end{aligned} \right.
\end{equation*}
Let $(\RPlinoptH{K_w}, \RPlinoptd{K_w})$ be the optimizer of $\RPlin{K_w}$, which is a random variables because it depends on the estimate $\Wapprox{K_w}$ which depends on the multisample $\setbf{W}^{K_w}$.
Note that in this case the domain $\set{S}$ in \eqref{eq:polyhedral_set} can be set to $\set{S} := \Wapprox{K_w}$, without loss of generality.

The robust program $\RPlin{K_w}$ can be transformed into a mixed-integer program and solved using off-the-shelf solvers, by dualizing the robust constraint, as in \cite{frick2017b}. This is possible, because $\setbf{K}_i$ are polyhedra and $\gamma_{\varphi,i}, \gamma_{s,i}$ are affine in $\mathbf{w}$.
Furthermore, the construction of $\Wapprox{K_w}$ allows us to give the following probabilistic guarantees for the solution $\RPlinoptH{K_w}$ of $\RPlin{K_w}$, relating it to the original robust program $\RP$.
\begin{theorem} \label{thm:robustlin}
	Let $N$ be the planning horizon, $n_w$ the dimension of the uncertainty and $m$ the number of inequalities of $\Wapprox{K_w}$. Let $\epsilon, \beta \in (0,1)$ be the desired violation and confidence parameter, respectively. For $K_w$ such that
	\begin{equation*}
		K_w \geq \frac{e}{e-1}\frac{1}{\epsilon}\big(\ln \left(\frac{1}{\beta}\right) + m(N+1)n_w + m - 1 \big)\,,
	\end{equation*}
	the solution $\RPlinoptH{K_w}$ of $\RPlin{K_w}$ satisfies
	\begin{equation*}
		\mathbb{P}^{K_w}[V_\varphi(\RPlinoptH{K_w}) > \epsilon] \leq \beta \text{ and }
		\mathbb{P}^{K_w}[V_s(\RPlinoptH{K_w}) > \epsilon] \leq \beta\,.
	\end{equation*}
\end{theorem}
\begin{proof}
	The proof follows by \reflem{lem:violation_probs} and construction of $\Wapprox{K_w}$ according to \cite{margellos2014}. We have that
	\begin{align*}
		\mathbb{P}^{K_w}[V_\varphi(\RPlinoptH{K_w}) > \epsilon] &\leq \mathbb{P}^{K_w}[\widetilde{V}_\varphi(\RPlinoptH{K_w}, \RPlinoptd{K_w}) > \epsilon ]  \\& \leq \mathbb{P}^{K_w}[\mathbf{w} \notin \setbf{W}[K_w] \,]\leq  \beta\,.
	\end{align*}
	The proof for $V_s(\RPlinoptH{K_w})$ follows in the same way.
\end{proof}

Note that \refrem{rem:tighter_bounds} holds analogously for \refthm{thm:robustlin}.

\begin{algorithm}[h]
	\caption{Policy synthesis, linear dependence on $\mathbf{w}$} \label{alg:synthesisLinear}
  	\begin{algorithmic}[1]
   		\Require{Piece-wise affine policy parametrization $\mathbf{H} \slimcdot \kappa(\mathbf{w})$, parametrization of $\delta$}
   		\Require{Violation and confidence parameters $\epsilon, \beta \in (0,1)$}
   		\State $K_w$ according to \refthm{thm:robustlin}
   		\State obtain multisample $\setbf{W}^{K_w}$ \Comment{e.g. from historical data}
   		\State $\Wapprox{K_w} \gets$ estimate from $\setbf{W}^{K_w}$ using \cite{margellos2014} 
   		\State $\RPlinoptH{K_w} \gets$ solve $\RPlin{K_w}$ using $\Wapprox{K_w}$
   		\State \Return $\RPlinoptH{K_w}$
  	\end{algorithmic}
\end{algorithm}

We outline the synthesis procedure in \refalg{alg:synthesisLinear}. The resulting solution enjoys the desired probabilistic guarantees for $\RP$ provided in \refthm{thm:robustlin}. These guarantees are exactly equivalent to the guarantees for $\SPapprox{\mathbb{K}}$ when $\epsilon = \epsilon_\varphi = \epsilon_s$ and $\beta = \beta_\varphi = \beta_s$.

\section{Case studies} \label{sec:numerics}
To illustrate our theoretical results and the practical performance of the proposed approach, we examine two simple motion planning case studies.
In case study~\ref{case_study_1}, we consider dealing with a turning truck. This illustrates the nonlinear dependence of the LTL specification on the uncertainty and shows that the policies obtained from \refalg{alg:synthesis} avoid the truck, satisfying the probabilistic guarantees of \refthm{thm:1}. 
In case study~\ref{case_study_2} we consider an overtaking maneuver. In this case, the dependence on $\mathbf{w}$ is linear, which allows us to find piecewise affine control policies using \refalg{alg:synthesisLinear} with significantly reduced computational effort compared to \refalg{alg:synthesis}. We demonstrate that the performance of the policy improves with a higher number of partition elements. The advantage of using feedback policies is illustrated in both of the case studies.
The code used for the two case studies is made available on \href{https://github.com/maryamka2018/HSCC2018Code}{github:maryamka2018}.

The controlled car is modeled as a double-integrator:
\begin{equation} \label{eqn:doubleint}
	\begin{bmatrix}\dot{x}_1 & \dot{x}_2\end{bmatrix}^\transp = \begin{bmatrix}x_3 & x_4\end{bmatrix}^\transp\,, \text{ and } \begin{bmatrix}\dot{x}_3 & \dot{x}_4\end{bmatrix}^\transp = \begin{bmatrix}u_1 & u_2\end{bmatrix}^\transp,
\end{equation}
where the state $x \in \reals{4}$ contains the position $(x_1,x_2)$ of the car and the corresponding velocities $(x_3,x_4)$. 
The accelerations $u = (u_1,u_2)$ are the control inputs and are subject to coupling constraints $\set{U}:= \{  (u_1,u_2) \in \reals{2} \sep{} \| \big[ \begin{smallmatrix}3.5 & 0 \\ 0 & 5 \end{smallmatrix}\big]^{-1} (u+\big[\begin{smallmatrix}\unitfrac[6.5]{m}{s^2} \\ \unitfrac[0]{m}{s^2}\end{smallmatrix}\big]) \|_1 \leq 1 \}$. Similarly, the velocities $x_3,x_4$ are coupled through the constraint set
\begin{equation} \label{eq:velocityConstraints}
	\set{X}:=\{ x \in \reals{4} \sep{}  \| \big[\begin{smallmatrix}v_1& 0 \\ 0 & v_2 \end{smallmatrix}\big]^{-1} \big(\big[\begin{smallmatrix} x_3 \\ x_4 \end{smallmatrix}\big]-\bar{v}\big) \|_1 \leq 1 \}\,,
\end{equation}
where $\bar{v}$, $v_1$ and $v_2$ depend on the speed limits for each task. Lane constraints will be included in $\set{X}$. The car has length $\unit[4.5]{m}$ and width $\unit[2]{m}$. In both case studies they are added to the obstacle dimensions and the car is treated via a point model.
A discrete-time version of \eqref{eqn:doubleint} with sampling time $T_s = 0.4$~seconds and a planning horizon of $N=10$ ($4$~seconds) is used for both case studies.
All computations were carried out on an Intel~i5 CPU at $\unit[2.8]{GHz}$ with $\unit[8]{GB}$ of memory, using YALMIP \cite{lofberg2004} and CPLEX \cite{cplex} as a mixed-integer problem solver. We used the default feasibility tolerance of $10^{-6}$, which was also used to compute the empirical violation probabilities.

\subsection{Avoiding a turning truck}\label{case_study_1}
The first example is illustrated in \reffig{numerics:fig:truck_time_evolution}. The car is driving with an initial forward velocity of $\unitfrac[50]{km}{h}$ on the right lane of a two-way street. A truck is coming from the other direction starting to turn left into a side street. 
The trajectory of position and orientation, \emph{the pose}, of the truck is uncertain. It follows unicycle dynamics,
\begin{equation} \label{eqn:unicycle}
	\begin{bmatrix}\dot{y}_1 & \dot{y}_2 & \dot{\theta} \end{bmatrix}^\transp =  \begin{bmatrix} v_t \cos(\theta) & v_t\sin(\theta) & \omega \end{bmatrix}^\transp\,,
\end{equation}
with a constant forward velocity $v_t = \unitfrac[22]{km}{h}$ and known initial pose $(\unit[44]{m},\unit[1.75]{m},0^\circ)$. The angular velocity $\omega$ enters through a zero-order hold and at time $k$ follows the distribution
\begin{equation*}
	\omega_k \sim \unitfrac[\uniform\big(\frac{\pi-0.66}{2(N+1)} , \min\big\{ \frac{\pi + 0.66}{2(N+1)}, \frac{\pi}{2} - \sum_{j=0}^{k-1} \omega_j \big\} \big)]{rad}{T_s}\,.
\end{equation*}
This means that the truck completes a $90^\circ$ counter-clockwise turn by the end of the planning horizon $N$. The disturbance $w_k = [w_{1,k}, w_{2,k}, w_{3,k}]^\transp := [y_{1,k}, y_{2,k}, \theta_k]^\transp$ is therefore \emph{correlated} over time and the support set $\setbf{W}$ is non-convex.

The goal is to maximize the expected terminal forward position of the car, i.e., to maximize $\mathbb{E}_{\mathbf{w}} [ x_{1,N}]$, while satisfying the LTL specification $\varphi:= \always \lnot p_{truck}$. The atomic proposition $p_{truck}$ encodes hitting the truck and is defined as in \refexa{exa:prop}, with $b = [\unit[9]{m},\unit[2.5]{m}]^\transp$. The uncertainty enters nonlinearly into the atomic proposition, due to the rotation matrix $\mathcal{R}(w_3)$.
Finally, we enforce state-input constraints, which include lane limits, velocity constraints \eqref{eq:velocityConstraints} with $\bar{v} = [\unitfrac[40]{km}{h}, \unitfrac[0]{km}{h}]^\transp $, $v_1= \unitfrac[40]{km}{h} $, and $v_2 = \unitfrac[20]{km}{h}$ and acceleration constraints $\set{U}$.

We consider affine feedback policies $\mathbf{u}(\mathbf{w}) = \mathbf{H} \cdot [\begin{smallmatrix}1 \\ \mathbf{w}\end{smallmatrix}]$ and a constant parametrization of $\delta$, as in \eqref{eq:piecewise_delta} with $P_\delta=1$. To reduce the computational burden we enforce a block-banded structure for $\mathbf{H}$ via $\set{H}$, so that the inputs at every time step depend only on the past three poses of the truck. The desired violation parameters are $\epsilon_\varphi = 0.05$, $\epsilon_s = 0.3$ and the confidence parameters are $\beta_\varphi = \beta_s = 10^{-3}$. We have obtained $K=6321$ i.i.d. truck trajectories which are allocated as $K_\varphi = 5454$, $K_s = 867$ according to \refthm{thm:1} and \refrem{rem:tighter_bounds}, and ensure the desired probabilistic guarantees.
The cost $J(\mathbf{H})$ is formulated using the sample average and is linear in $\mathbf{H}$. The scenario program $\SPapprox{(K_\varphi,K_s)}$ is a mixed-integer linear optimization problem with 253'016 continuous variables, 40 binary variables and 629'580 constraints. Instead of using the form~\eqref{eq:state_evolution}, we introduced optimization variables for the states. This led to a larger but sparser problem, requiring less time to solve. 
We solved 15 different instances of $\SPapprox{(K_\varphi,K_s)}$, each for a different multisample, taking 63~minutes on average. The resulting policies lead to trajectories where the car breaks just enough to avoid hitting the truck, in spite of its uncertain motion. This is illustrated in \reffig{numerics:fig:truck_time_evolution} for the policy $\SPapproxoptH{(K_\varphi,K_s)}$ of the first instance and a new random realization $\mathbf{w}$.

\begin{figure}
	\centering
	\begin{subfigure}[b]{0.8\columnwidth}
		\centering
		\includegraphics[width=\linewidth]{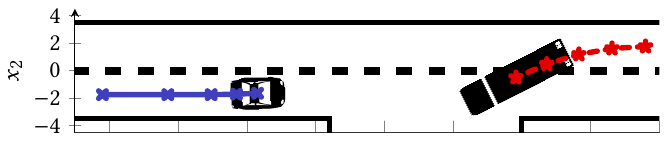}
	\end{subfigure}\\
	\vspace*{-0.5em}
	\begin{subfigure}[b]{0.8\columnwidth}
		\centering
		\includegraphics[width=\linewidth]{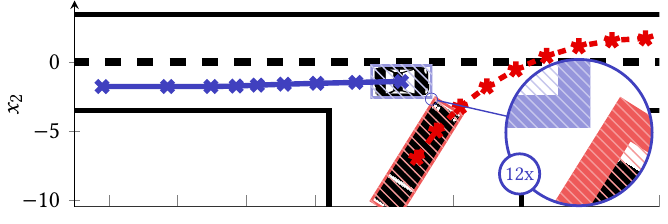}
	\end{subfigure}\\
	\vspace*{-0.5em}
	\hspace*{0.35em}
	\begin{subfigure}[b]{0.815\columnwidth}
		\centering
		\includegraphics[width=\linewidth]{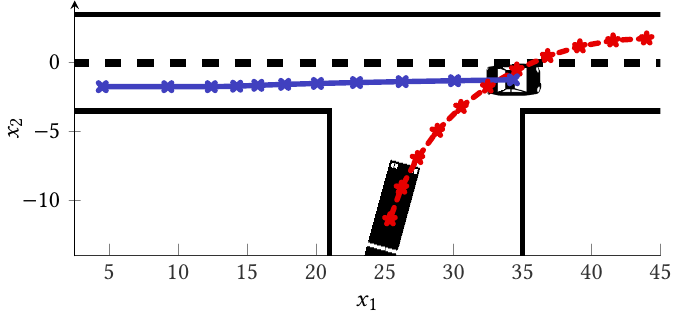}
	\end{subfigure}
	\vspace{-1em}
	\caption{The controlled car ({\color{blue}solid}) reacting to a new random truck trajectory ({\color{red}dashed}) using the feedback policy $\SPapproxoptH{(K_\varphi,K_s)}$, pictured are time steps $k=4,8$ and $10$.}
	\label{numerics:fig:truck_time_evolution}
\end{figure}

\begin{figure}
	\centering
	\vspace*{-1em}
	\begin{subfigure}[b]{0.73\columnwidth}
		\centering%
		\includegraphics[width=\linewidth]{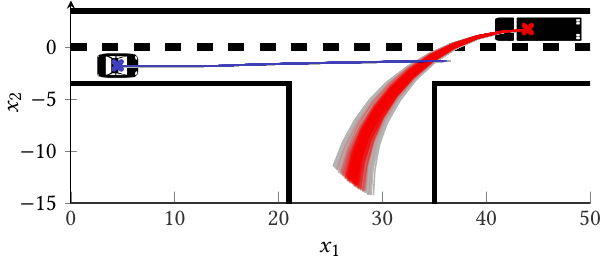}
		\caption{}
		\label{numerics:fig:use_of_feedback}
	\end{subfigure}
	\begin{subfigure}[b]{0.25\columnwidth}
		\centering
		\includegraphics[width=\linewidth]{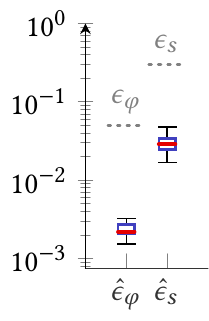}
		\caption{}
		\label{numerics:fig:empirical_violations}
	\end{subfigure}
	\vspace{-1em}
	\caption{a) Reactions to 10'000 new random truck trajectories. b) Empirical violation probabilities over 15 instances of $\SPapprox{(K_\varphi,K_s)}$, {\color{red}median}, $25$-th to $75$-th {\color{blue}percentile}, guaranteed violation probabilities $\epsilon_\varphi$, $\epsilon_s$ ({\color{gray}dotted}).}
	\vspace{-1em}
\end{figure}

We have evaluated the policy $\SPapproxoptH{(K_\varphi,K_s)}$ for 10'000 new random truck trajectories. The resulting car trajectories are given in \reffig{numerics:fig:use_of_feedback} and show that the car adapts its speed to the pose of the truck. The terminal position of the car depends on the truck trajectory and varies between $\unit[34.3]{m}$ and $\unit[36.6]{m}$.
For each of the 15 instances the empirical violation probabilities $\hat{\epsilon}_\varphi$, $\hat{\epsilon}_s$ are computed for the same 10'000 new random truck trajectories. Their distribution is illustrated in \reffig{numerics:fig:empirical_violations}. Indeed, as expected from \refthm{thm:1}, they conservatively satisfy $\epsilon_\varphi$ and $\epsilon_s$.  

\subsection{Safe overtaking}\label{case_study_2}
In the second task, the car is driving at $\unitfrac[100]{km}{h}$ on the left lane of a three-lane road and wants to overtake a truck which is driving in the middle lane. The position of the truck is uncertain and it may change lanes. 
The truck follows single-integrator dynamics
\begin{equation} \label{eqn:singleint}
	\begin{bmatrix}\dot{y}_1 & \dot{y}_2 \end{bmatrix}^\transp =  \begin{bmatrix} v_f  & v_l \end{bmatrix}^\transp,
\end{equation}
where $(y_1,y_2)$ is the uncertain position of the truck. The forward velocity $v_f = \unitfrac[100]{km}{h}$ is known and constant. Furthermore, the initial position is uncertain, with $y_{1,0} \sim \uniform (\unit[13.5]{m},\unit[17.5]{m})$ and $y_{2,0} \sim T_s\uniform (\minus\unitfrac[3.75]{km}{h}, \unitfrac[3.75]{km}{h}) - \unit[1.75]{m}$.
Moreover, the lateral velocity $v_l$, for $k=0,\ldots,N$, follows the distribution 
\begin{equation*}
	 v_{l,k} \sim \begin{cases}
	 	\uniform (\unitfrac[0]{km}{h}, \unitfrac[3.75]{km}{h} ) & \text{if } y_{2,0} \geq \unit[-1.75]{m}\,, \\
	 	\uniform (\minus\unitfrac[3.75]{km}{h},\unitfrac[0]{km}{h} ) & \text{otherwise}\,,
	 \end{cases}
\end{equation*}
This means that the truck does not change the direction of its lateral motion, once it has decided whether to go left or right.
Hence, as in the previous example, the disturbance states $w_k = [w_{1,k}, w_{2,k}]^\transp := [ y_{1,k}, y_{2,k}]^\transp$ are \emph{correlated} in time and the support set $\setbf{W}$ is non-convex.
As in the previous task, the car wants to maximize the expected terminal forward position, which is approximated by the sample average, while satisfying the LTL specification $\varphi:= \always \lnot p_{truck}$. In this case, $p_{truck}$ is described as in \refsec{case_study_1} without rotation, i.e., for $w_3 = 0^\circ$.
Lane limits, velocity constraints \eqref{eq:velocityConstraints} with $\bar{v} = [\unitfrac[110]{km}{h}, \unitfrac[0]{km}{h}]^\transp $, $v_1= \unitfrac[20]{km}{h} $, and $v_2 = \unitfrac[20]{km}{h}$, and acceleration constraints $\set{U}$ are also enforced. The allowable violation is $\epsilon = 0.05$ with confidence parameters $\beta = 10^{-3}$.

We consider piecewise-affine policies, as in \refexa{example:piecewise_affine}, and restrict the parametrization of $\delta$ to be defined over the same partition as the policy $\mathbf{u}(\cdot)$ for simplicity. Therefore, the problem depends linearly on $\mathbf{w}$ and falls into the class considered in \refsec{subsec:linear_in_w}.

\begin{figure*}
	\centering
	\begin{subfigure}[b]{0.33\textwidth}
		\centering
		\includegraphics[width=\linewidth]{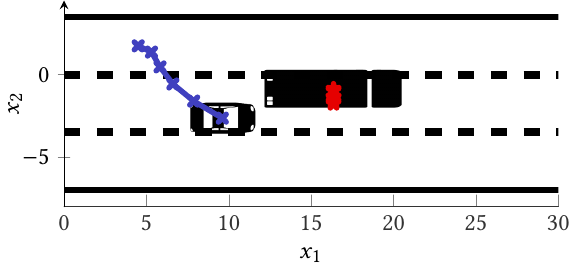}
	\end{subfigure}
	\begin{subfigure}[b]{0.307\textwidth}
		\centering
		\includegraphics[width=\linewidth]{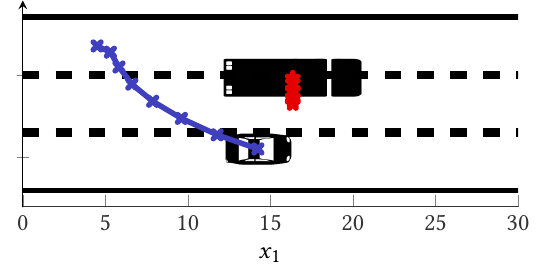}
	\end{subfigure}
	\begin{subfigure}[b]{0.307\textwidth}
		\centering
		\includegraphics[width=\linewidth]{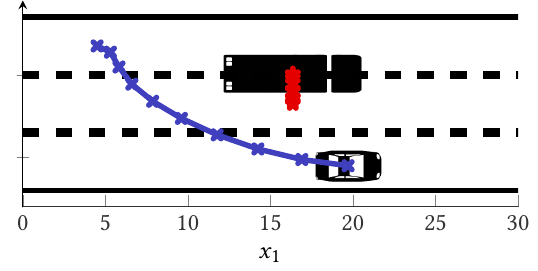}
	\end{subfigure}
	\vspace*{-1em}
	\caption{The controlled car ({\color{blue}solid}) reacting to a new random truck trajectory ({\color{red}dashed}) using feedback policy $\RPlinoptH{K_w}$, pictured at time steps $k=5,7$ and $9$. The coordinate system is relative to a reference frame moving with $\unitfrac[100]{km}{h}$.}
	\label{numerics:fig:overtake_time_evolution}
\end{figure*}

\begin{figure}
	\centering
	\begin{subfigure}[b]{0.7\columnwidth}
		\centering%
		\includegraphics[width=\linewidth]{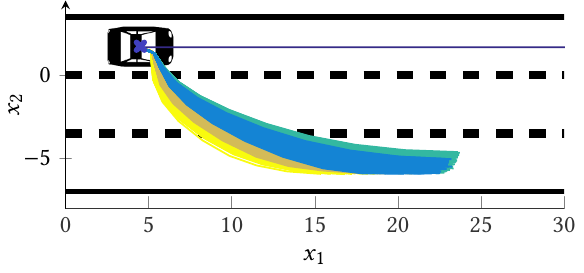}
		\vspace*{-1.7em}
		\caption{}
		\label{numerics:fig:overtake_1000_new_realiz}
	\end{subfigure}
	\hspace*{-1.06em}
	\begin{subfigure}[b]{0.318\columnwidth}
		\centering
		\includegraphics[width=\linewidth]{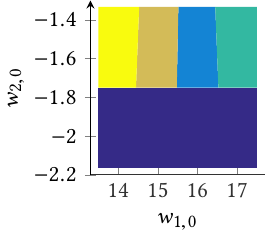}
		\vspace*{-1.7em}
		\caption{}
		\label{numerics:fig:overtake_partition}
	\end{subfigure}
	\vspace*{-2.3em}
	\caption{a) Car trajectories for 10'000 new random truck trajectories. The colors correspond to the different pieces of the policies for the partition in b).}
	\vspace*{-1em}
\end{figure}

Following \refalg{alg:synthesisLinear}, we estimate the support set $\setbf{W}$ as a union of two polyhedral sets $\setbf{W}[K_w] := \setbf{W}_1[K_w] \cup \setbf{W}_2[K_w]$, where the two sets are separated by a hyperplane $\{\mathbf{w} \sep{} w_{2,1} \geq w_{2,0} \}$ that partitions $\reals{(N+1)n_w}$. This separation can be identified from samples, e.g. using support vector machines, and it distinguishes the truck going left from it going right. More precisely, we define $\setbf{W}_1[K_w]:= \{ \mathbf{w} \sep{} w_{2,1} \leq w_{2,0}\,,\, \underbar{\mathbf{w}}_1 \leq \mathbf{w} \leq \overbar{\mathbf{w}}_1\}$ and $\setbf{W}_2[K_w]:= \{ \mathbf{w} \sep{} w_{2,1} > w_{2,0}\,,\, \underbar{\mathbf{w}}_2 \leq \mathbf{w} \leq \overbar{\mathbf{w}}_2\}$. We select $K_w=2381$ according to \refthm{thm:robustlin} and \refrem{rem:tighter_bounds}, and identify the bounds $\underbar{\mathbf{w}}_1$, $\overbar{\mathbf{w}}_1$, $\underbar{\mathbf{w}}_2$ and $\overbar{\mathbf{w}}_2$ using the scenario approach, as in \cite{margellos2014}.

Based on the estimate $\setbf{W}[K_w]$ we define the partition $\{\set{K}_i\}_{i=1}^P$ for the policy. For $P = 1$, we choose $\set{K}_1$ as the smallest box containing $\setbf{W}[K_w]$. For $P=2$ we select $\set{K}_1 := \setbf{W}_1[K_w]$ and $\set{K}_2 := \setbf{W}_2[K_w]$. Partitions with more elements are generated by further partitioning $\setbf{W}_1[K_w]$ and $\setbf{W}_2[K_w]$. This can be achieved following the iterative partitioning scheme in \cite{postek2016}. For each element of the partition we find the most binding scenarios $\mathbf{w}$ by solving a linear program for each constraint row. Then, we divide the element into two pieces using a splitting hyperplane that separates the most distant binding scenarios using \cite[Heuristic~1]{postek2016}. In general, \emph{non-anticipativity} constraints need to be added to enforce causality of the policy $\mathbf{u}(\cdot)$. However, we partition using 1-splitting hyperplanes \cite[Section~5.1]{postek2016} which ensures causality without additional constraints. In effect, only the uncertain initial position $w_0=[w_{1,0}, w_{2,0}]$ of the truck is partitioned. Furthermore, only $\setbf{W}_2[K_w]$ needs to be partitioned, since the policy cannot be improved by further partitioning $\setbf{W}_1[K_w]$.
We consider partitions with $P=3,5,9$ elements, where $\set{K}_1 := \setbf{W}_1[K_w]$ and $\{\set{K}_i\}_{i=2}^P$ is a partition of $\setbf{W}_2[K_w]$ constructed as outline above. The estimate $\setbf{W}[K_w]$ and the partitions are obtained in $0.31$, $0.31$, $2.73$, $7.99$, and $28.49$ seconds for $P = 1,2,3,5$ and $9$, respectively.

\begin{figure}
	\centering
	\includegraphics[width=0.45\columnwidth]{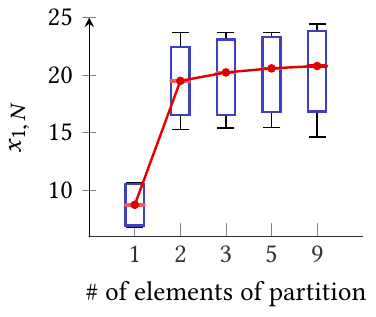}
	\vspace*{-1em}
	\caption{Distributions of terminal forward position $x_{1,N}$ for 10'000 new trajectories of the truck moving left and policies with $P=1,2,3,5$, and $9$ elements. The mean ({\color{red}red}) and $5$-th to $95$-th percentile ({\color{blue}blue}).}
	\label{numerics:fig:overtake_objective}
\end{figure}

For a policy with $P=5$ we solve $\RPlin{K_w}$, a linear program with 37'680 continuous variables, 200 binary variables, and 54'930 constraints. The optimal policy $\RPlinoptH{K_w}$ was obtained in $19.18$~seconds. In contrast, \refalg{alg:synthesis} exceeded the available memory.
The performance is illustrated in \reffig{numerics:fig:overtake_time_evolution} for a new trajectory of the truck and shows a successful overtaking maneuver. Note that we depict the car relative to a reference frame moving with $\unitfrac[100]{km}{h}$.
We additionally evaluated the policy $\RPlinoptH{K_w}$ for 10'000 random truck trajectories. The empirical violation probability is $\hat{\epsilon} = 9 \cdot 10^{-4}$. The resulting trajectories of the car are shown in \reffig{numerics:fig:overtake_1000_new_realiz}, illustrating how the car adapts to different truck positions. The partitioning of $\setbf{W}[K_w]$ is shown in \reffig{numerics:fig:overtake_partition}. When the truck moves right, $\RPlinoptH{K_w}$ leads to a trivial maneuver with terminal position $\unit[32.6]{m}$ for any realization of the uncertainty. When the truck moves left, the terminal position achieves values between $\unit[15.3]{m}$ and $\unit[23.7]{m}$, depending on the specific uncertainty realization.

\reffig{numerics:fig:overtake_objective} shows the distributions of the terminal positions of the car for 10'000 trajectories of the truck moving left, as a function of $P$. The average terminal position, as well as the $5$-th and $95$-th percentile marks are shown and increase with larger number of partitions $P$.
This illustrates that richer policies allow the car to adapt better to the possible scenarios and improve the cost. Note that for $P=1$ the only safe maneuver for the car is to not overtake, while for $P \geq 2$ the car can overtake safely. Moreover, there does not exist an open-loop policy that safely overtakes the truck in most cases, due to the uncertainty of the motion of the truck.

\section{Conclusions}
We have addressed the problem of synthesizing disturbance feedback policies for systems satisfying LTL specifications with uncertainty in the dynamics and the specification. In particular, we have formulated an optimization problem where the specification and operational constraints are to be fulfilled robustly, i.e., for all disturbances. We have proposed a data-driven approach that only requires sample realizations of the uncertain variables to obtain policies that satisfy the constraints with probabilistic guarantees.

The policies can be computed as solutions to mixed-integer programs. We have provided bounds on the number of samples needed such that the probability of violating the system and specification constraints are separately limited with given confidence parameters. Moreover, we have considered the special case where the disturbance enters linearly. For this case, we estimate the support set of the uncertainty and solve a simpler robust mixed-integer optimization problem that provides similar probabilistic guarantees.

We have demonstrated the performance of the method in two autonomous driving examples. Theses case studies show the applicability of the scenario approach to provide probabilistic constraint satisfaction guarantees based on sampled uncertainty data. Furthermore, we have illustrated the benefit of incorporating disturbance feedback to deal with unseen instances of uncertainty in real-time, improving feasibility and performance. 

\appendix

\section{Proof of Proposition~4.4} \label{app:proofs}
\begin{proof}[Proof of \refprop{prop:infinite_helly_dimension}]
Inspired by \cite[Fig.~8]{campi2018}, we will show that there exist robust programs of form $\RP$ such that for any given $K_\varphi$, all sampled constraints $\eta_\varphi(\cdot, \mathbf{w}^{(k)})$, $k=1,\ldots,K_\varphi$, may be supporting for $\SP{\mathbb{K}}$. Then, by \refdef{def:supportdim} we have $\zeta_\varphi = K_\varphi$.
We consider a particular robust program with $\mathbf{H}=[h_1, h_2] \in \set{H} := \{ \mathbf{H} \in \reals{2} \sep{} 0 \leq h_1 \leq 1 \}$ and linear objective function $J(\mathbf{H}) = h_2$. There is no constraint $\gamma_s(\mathbf{H},\mathbf{w})$, we therefore set $K_s = 0$ without loss of generality. We define the constraint function
\begin{equation*}
	\eta_\varphi(\mathbf{H},\mathbf{w}) := \begin{cases} h_1 - w_1 -h_2 & \text{if } h_1 \leq w_1 + w_2\,,\\
	w_1 + 2w_2 - h_1 -h_2 & \text{if } h_1 \geq w_1 + w_2\,, \end{cases}
\end{equation*}
which is piecewise-affine concave in $\mathbf{H}$.
The uncertainty $\mathbf{w} = (w_1,w_2) \in \reals{2}$ is distributed uniformly over $\setbf{W} := \{(w_1,w_2) \in \reals{2} \sep{} -2 \leq w_i \leq 2 , i=1,2 \}$.
For each $K_\varphi$ we can define the following set of multisamples
\begin{align*}
	\bm{\mathfrak{W}}^{K_\varphi} := \big\{ \setbf{W}^{K_\varphi} \in \setbf{W} \sep{} & w_1^{(k)} \in \{\tfrac{4k-5}{4K_\varphi}\} + ( -\tfrac{1}{16 K_\varphi} , \tfrac{1}{16 K_\varphi})\,,\\[-0.2em]
	& \hspace*{-1em} w_2^{(k)} \in \{\tfrac{1}{K_\varphi}\} + (0 , \tfrac{1}{8K_\varphi}) \text{ for } k=1,\ldots,K_\varphi \big\}\,.
\end{align*}
Notice that $\bm{\mathfrak{W}}^{K_\varphi}$ is non-zero measure in $\setbf{W}$. Furthermore, it can be verified that for any $\setbf{W}^{K_\varphi} \in \bm{\mathfrak{W}}^{K_\varphi}$ all the constraints of $\SP{\{K_\varphi,0\}}$ are supporting.
The feasible region corresponding to one such multisample is depicted in \reffig{Figure_PWA_support_constraints} for $K_\varphi=3$. In the three subfigures, we illustrate that all constraints are supporting and indicate the new optimizer obtained when each of them is removed.
This means that for any $K_\varphi \geq 1$, $\SP{(K_\varphi,0)}$ has support dimension $\zeta_\varphi = K_\varphi$ with non-zero probability. This concludes the proof.
\begin{figure}[ht]
	\centering
	\begin{subfigure}[b]{0.6\columnwidth}
		\includegraphics[width=\linewidth]{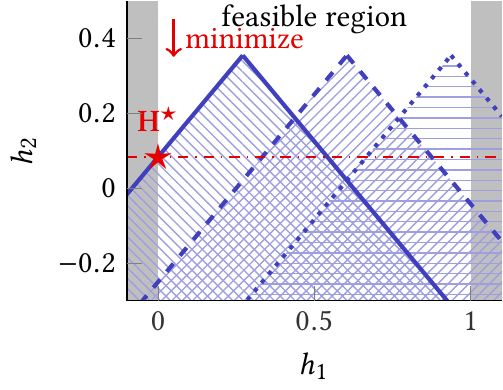}
	\end{subfigure}
	\begin{minipage}[b]{0.2\columnwidth}
	\centering
	\begin{subfigure}[b]{\columnwidth}
		\centering
		\includegraphics[width=\linewidth]{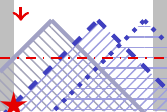}
	\end{subfigure}\\
	\begin{subfigure}[b]{\columnwidth}
		\centering
		\includegraphics[width=\linewidth]{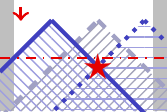}
	\end{subfigure}\\
	\begin{subfigure}[b]{\columnwidth}
		\centering
		\includegraphics[width=\linewidth]{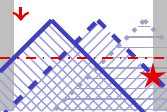}
	\end{subfigure}
	\vspace*{0.33em}
	\end{minipage}
	\vspace*{-1em}
	\caption{Feasible region of $\SP{(3,0)}$ with a multisample $\setbf{W}^{3} \in \mathbb{\mathbf{W}}^{3}$, illustrating that all constraints are supporting.}
	\label{Figure_PWA_support_constraints}
	\vspace{-1em}
\end{figure}
\end{proof}

\bibliographystyle{ACM-Reference-Format}
\bibliography{biblio} 

\end{document}